\documentclass[10pt,leqno]{amsart}
\usepackage{indentfirst, amssymb,amsmath,amsthm}
\newtheoremstyle{case}{}{}{}{}{}{:}{ }{}
\newtheoremstyle{subcase}{}{}{}{}{}{:}{ }{}

\newtheorem*{theoA}{Theorem A}
\newtheorem*{theoB}{Theorem B}
\newtheorem*{theoC}{Theorem C}
\newtheorem*{theoD}{Theorem D}
\newtheorem*{theoE}{Theorem E}
\newtheorem*{theoF}{Theorem F}
\newtheorem*{theoG}{Theorem G}
\newtheorem*{theoH}{Theorem H}
\newtheorem*{theoI}{Theorem I}

\newtheorem{theo}{Theorem}[section]
\newtheorem{lem}{Lemma}[section]
\newtheorem{cor}{Corollary}[section]
\newtheorem{note}{Note}[section]
\newtheorem{exm}{Example}[section]
\newtheorem{defi}{Definition}[section]
\newtheorem{rem}{Remark}[section]

\newcommand{\ol}{\overline}

\numberwithin{subcase}{case}
\newcommand{\be}{\begin{equation}}
\newcommand{\ee}{\end{equation}}
\newcommand{\beas}{\begin{eqnarray*}}
	\newcommand{\bea}{\begin{eqnarray}}
	\newcommand{\eea}{\end{eqnarray}}
	\newcommand{\eeas}{\end{eqnarray*}}
\newcommand{\lra}{\longrightarrow}
\newcommand{\bd}{\begin{doublespacing}}
	\newcommand{\ed}{\end{doublespacing}}
\usepackage{amssymb,setspace,amsbsy,indentfirst}
\numberwithin{equation}{section}
\renewcommand{\vline}{\mid}
\begin{document}
	\title[Uniqueness of $\boldmath{L}$ function with special class of meromorphic function .....]{ Uniqueness of $\boldmath{L}$ function with special class of meromorphic function under restricted sharing of sets}
	\date{}
	\author{Abhijit Banerjee\;\;\; and \;\;\;Arpita Kundu}
	\date{}
	\address{ Department of Mathematics, University of Kalyani, West Bengal 741235, India.}
	\email{abanerjee\_kal@yahoo.co.in, abanerjeekal@gmail.com}
	
	\address{Department of Mathematics, University of Kalyani, West Bengal 741235, India.}
	\email{arpitakundu.math.ku@gmail.com}
	\maketitle
	\let\thefootnote\relax
	\begin{abstract}
		The purpose of the paper is to rectify a series of errors occurred in  \cite{Ban-Kundu_lmj}, \cite{Sahoo-Sarkar_al_i.cuza}, \cite{Yan-Li-Yi_Lith} for a particular situation. To get a fruitful solution and to overcome the issue, we introduce a new form of set sharing namely restricted set sharing, which is stronger than the usual one. We manipulate the newly introduced notion in this specific section of literature to resolve all the complications. Not only that we have subtly used the same sharing form to a well known unique range set \cite{Frank-Reinders_comp.var} to settle a long time unsolved problem.  
	\end{abstract}
	\section{introduction}	In the paper we adopt the standard notations of Nevanilinna theory of meromorphic functions as explained in \cite{W.K.Hayman_64}.
	We write   $\mathbb{\overline{C}}=\mathbb{C}\cup\{\infty\}$, $\mathbb{C^*}=\mathbb{C}\setminus\{0\}$ and $ \mathbb{\ol N}=\mathbb{N}\cup \{0 \}$, where $\mathbb{C}$ and $\mathbb{N}$  denote the set of all complex numbers and natural numbers respectively.  We begin by recalling the following definition from the literature. 
	
	\begin{defi}\cite{Ban-Mallick_cmft}
		For a  non-constant meromorphic function $f$ and  $a\in{\mathbb{C}}$, let  $E_{f}(a)=\{(z,p)\in\mathbb{C}\times\mathbb{N}: f(z)=a\; with\; multiplicity\; p\}$ $ \left(\ol  E_{f}(a)=\{(z,1)\in\mathbb{C}\times\mathbb{N}: f(z)=a\}\right)$. Then we say  $f$, $g$ share the value $a$ CM(IM) if $E_{f}(a)=E_{g}(a)$$\left( \ol E_{f}(a)=\ol E_{g}(a)\right).$ For $a=\infty$, we define $E_f(\infty) := E_{1/f}(0)$ $\left( \ol E_f(\infty) :=\ol  E_{1/f}(0)\right)$.
	\end{defi}
	\begin{defi} \cite{Ban-Mallick_cmft}
		For a  non-constant meromorphic function $f$ and  $S\subset\overline{\mathbb{C}}$, let $E_{f}(S)=\bigcup_{a\in S}\{(z,p)\in\mathbb{C}\times\mathbb{N}: f(z)=\text {a\; with\; multiplicity\; p}\}$ $\left(\ol  E_{f}(S)=\bigcup_{a\in S}\{(z,1)\in\mathbb{C}\times\mathbb{N}: f(z)=a\}\right) $. Then we say  $f$, $g$ share the set $S$ CM(IM) if $E_{f}(S)=E_{g}(S)$ $\left(\ol  E_{f}(S)=\ol E_{g}(S)\right)$.
	\end{defi} 
	\begin{defi}\cite{Lahiri_Nagoya(01)}
		Let k be a nonnegative integer or infinity. For $a\in\overline{\mathbb{C}}$  we
		denote by $E_k(a; f)$ the set of all a-points of f, where an $a$-point of multiplicity $m$ is counted
		$m$ times if $m\leq k$ and $k + 1$ times if $m > k$. If $E_k(a; f) = E_k(a; g)$, we say that $f$, $g$ share
		the value a with weight $k$.
	\end{defi}
	We write $f$, $g$ share $(a,k)$ to mean that $f$, $g$ share the value $a$ with weight $k$. Clearly if
	$f$, $g$ share $(a,k)$, then $f$, $g$ share $(a,p)$ for any integer $p$, $0\leq p < k$. Also we note that $f$, $g$
	share a value $a$ IM or CM if and only if $f$, $g$ share $(a,0)$ or $(a,\infty)$ respectively.
	\begin{defi}\cite{Lahiri_Nagoya(01)}
		For $S\subset\mathbb{\ol C}$ we define $E_f(S,k)=\cup_{a\in S}E_k(a;f)$, where k is a non-negative integer $a\in S$ or infinity. Clearly $E_f(S)=E_f(S,\infty)$ and $\overline{E}_f(S)=E_f(S,0)$, i.e $f$ and $g$ share $S$ CM (IM) implies $f$ and $g$ share $(S,\infty)\;((S,0))$.
	\end{defi}
	
	When $S$ contains only one element the definition coincides with the classical definition of value sharing.
	\par We assume that the readers are familiar with the standard notations of Nevanlinna theory such as the Nevanlinna characteristic function $T(r,f)$, the proximity function $ m(r,f)$, the counting function (reduced counting function) $N(r,\infty;f)$ ($\ol N(r,\infty;f)$) and so on, which are well explained in \cite{W.K.Hayman_64}. We use the symbol $\rho(f)$ to denote the order of a non constant meromorphic function $f$, which is defined as $$\rho(f)=\limsup_{r\lra\infty}\frac{\log^{+}T(r,f)}{\log r}.$$
	
	\par By $S(r,f)$ we mean any quantity satisfying $S(r,f)=O(\log(rT(r,f)))$ for all r possibly outside a set of finite Lebesgue measure. If $f$ is a function of finite order, then $S(r,f)=O(\log r)$ for all $r$. In this paper we consider $f$ is a non constant meromorphic function having finitely many poles in $\mathbb{C}$. Clearly $\ol N(r,\infty;f)=O(\log r)$.

	\par We know that the Riemann hypothesis is a conjecture which states that the Riemann zeta function $\zeta(s)=\sum_{n=1}^{\infty}\frac{1}{n^s}$ has all its zeros only at the negative even integers and complex numbers with real part $\frac{1}{2}$. 	The Riemann zeta function is  defined in the half-plane $Re(s) > 1$ by the absolutely convergent series $\zeta(s)=\sum_{n=1}^{\infty}\frac{1}{n^{s}}$
	and in the whole complex plane $\mathbb{C}$ by analytic continuation. 
	Dirichlet’s $L$-functions are natural extensions of the Riemann zeta function and so the Riemann hypothesis can be generalized by replacing the Riemann zeta function by $L$-functions. 
	\par The value distribution of an $L$-function $\mathcal{L}$ is defined as that of meromorphic function. Analogous to the definition of meromorphic function, for some $c\in \mathbb{C}\cup\{\infty\}$, it is about the roots of the equation $\mathcal{L}(s)=c$. Recently, a new trend among the researchers have been  to investigate the value distributions of $L$-functions (see  \cite{Graun-Grahl-Steuding}, \cite{Hu_LI_Can-16}, \cite{B.Q.LI-Proc.Am-10}, \cite{Li-Yi_Nachr}, \cite{Steuding-Sprin-07}). 
	The selberg class $\mathcal{S}$  of $L$-function is the set of Dirichlet series $\sum_{n=1}^{\infty}\frac{a(n)}{n^s}$ of complex variable $s$ satisfy four axioms given in \cite{selberg-92}.
	\par In 1920, Nevanlinna's five value theorem and four value theorem inspired many mathematicians to further investigate the situation under different perspective. 
	
	In 2007, Steuding [p.152, \cite{Steuding-Sprin-07} ] proved uniqueness problem of two $\mathcal{L}$ functions under reduced shared values as follows:
	\begin{theoA}
		If two $L$-functions $\mathcal{L}_{1}$ and $\mathcal{L}_{2}$ with $a(1)=1$ share a complex value $c\;(\not=\infty)$ CM, then $\mathcal{L}_{1}=\mathcal{L}_{2}$.\end{theoA}
	\begin{rem}
		Providing a counterexample, Hu and Li \cite{Hu_LI_Can-16} have pointed out that {\it{Theorem A}} is not true when c = 1. \end{rem} 
	
	Since $L$-functions possess meromorphic continuations, it is quite natural to investigate  to which extent an $L$-function can share values with an arbitrary meromorphic function. In 2010, providing a counter example, Li \cite{B.Q.LI-Proc.Am-10} observed that {\it{Theorem A}} is not true for an $L$-function and a meromorphic function.
	\begin{exm}
		For an entire function g, the functions $\zeta$ and $\zeta e^{g}$ share $(0,\infty)$, but $\zeta \not=\zeta e^{g}$ where $g$ is supposed not to vanish identically.
	\end{exm}
	However, considering two distinct complex values Li \cite{B.Q.LI-Proc.Am-10} was able to find the uniqueness result corresponding to a $L$-function $\mathcal{L}$ and a non-constant meromorphic function $f$.
	
	\par Now let us define the two polynomials $P(w)$, $P_{1}(w)$ as follows: $$P(w)=w^{n}+aw^{m}+b\;\; \text {and}\;\;P_{1}(w)=w^{n}+aw^{n-m}+b,$$ where $a$, $b\in\mathbb{C}\backslash\{0\}$ and $n$, $m$ be two positive integers such that gcd$(n,m)=1$. \par In view of {\it Lemma 4}, \cite{Ban-Kundu_lmj}, we see that both $P(w)$ and $P_{1}(w)$ can have at most one multiple zero. Next corresponding to the zeros of the polynomials $P(w)$, $P_{1}(w)$, let us define two sets $S$ and $S_{1}$ as follows: \bea\label{e1.1} S=\{w:P(w)=0\}=\{\alpha_{1},\alpha_{2},\ldots,\alpha_{l}\}\eea \bea\label{e1.2} S_{1}=\{w:P_{1}(w)=0\}=\{\beta_{1},\beta_{2},\ldots,\beta_{l}\},\eea where $n-1\leq l\leq n$. \par 
	We first recall the following definition:
	\begin{defi}\cite{Yi_comp.var(97)}
		Let $S\subset \mathbb{C}$ such that if two non-constant meromorphic functions $f$ and $g$ share the $S$ CM implies $f\equiv g$, then we say $S$ is a URSM.
	\end{defi}
	
	
	\par The polynomial corresponding to URSM having $n$ elements is called a PURSM (see p. 447, \cite{P.Li-C.C.Yang_kodai(95)}) and it always contain simple zeros. Probably Yuan-Li-Yi \cite{Yan-Li-Yi_Lith} were the first authors who dealt with such a set $S=\{w:P(w)=0\}$, where $P(w)$ can have a multiple zero and then investigated the uniqueness problem for both the cases $l=n$ and $l=n-1$. Corresponding to the sharing of one or two finite sets for a meromorphic function $f$ and an $L$-function $\mathcal{L}$,  Yuan-Li-Yi's \cite{Yan-Li-Yi_Lith} results demonstrated below actually originated from the question of Gross \cite{G.F_Springer(1977)}.	
	\begin{theoB}\cite{Yan-Li-Yi_Lith}
		Let $f$ be a meromorphic function having finitely many poles in $\mathbb{C}$ and let ${\mathcal{L} } $ be a non-constant $L$-function. Let $S$ be defined as in (\ref{e1.1}), where $n\geq 5$ and $n>m$ and $c\in\mathbb{C}\setminus S\cup\{0\}$. If $f$ and $\mathcal{L}$ share $(S,\infty)$ and $(c,0)$, then $f\equiv \mathcal{L}$.
	\end{theoB}
	\begin{theoC}\cite{Yan-Li-Yi_Lith}
		Let $f$ and $\mathcal{L}$ be a defined as in {\text Theorem C}.
		Let $S$ be defined as in (\ref{e1.1}) where $n>2m+4$. If $f$ and $\mathcal{L}$ share  $(S,\infty)$, then $f \equiv \mathcal{L}$.
	\end{theoC}
	We are not going to discuss the basic definitions or notations related to weighted sharing and others, as for those one can go through the previous papers \cite{Lahiri_comp.var}, \cite{Lahiri_Nagoya(01)}. \par Relaxing the CM sharing into finite weight, in 2020, Sahoo-Sarkar \cite{Sahoo-Sarkar_al_i.cuza} proved the following results.
	\begin{theoD}\cite{Sahoo-Sarkar_al_i.cuza} Let $S$ be defined as in {\it{Theorem B}} and $n>2m+4\;or\;2k+4$ \em{(}$k=n-m\geq 1$\em{)}. Let
		$f$ be meromorphic function having finitely many poles in $\mathbb{C}$ and let $\mathcal{L}$ be a
		non constant $L$-function. If $f$ and $\mathcal{L}$ share $(S,2)$ and $(c,0)$, then $f \equiv \mathcal{L}$.\end{theoD}
	\begin{theoE}\cite{Sahoo-Sarkar_al_i.cuza} Let $S$ be defined ame as in {\it{Theorem B}} and $n>2m+4$. Let
		$f$ be meromorphic function having finitely many poles in $\mathbb{C}$ and let $\mathcal{L}$ be a
		non constant $L$-function. If $f$ and $\mathcal{L}$ share $(S,2)$, then $f \equiv \mathcal{L}$\end{theoE}
	Very recently, the present authors pointed out a number of gaps in the paper of Sahoo-Halder\cite{Sahoo-Hal_20_Lith} and rectifying them presented the extended result in case of all gradations of sharing. Our result was as follows: 
	
	\begin{theoF}\cite{Ban-Kundu_lmj}
		Let $S$ be defined as in (\ref{e1.1}). Also let $f$ be a meromorphic function having finitely many poles in $\mathbb{C}$ and $\mathcal{L}$ be a non-constant $L$-function such that $E_{f}(S,s)=E_{\mathcal{L}}(S,s)$. If
		\\(i) $s\geq 2$ and $n>2m+4$, or if
		\\(ii) $s=1$ and $n>2m+5$, then $f\equiv \mathcal{L}$.
	\end{theoF}
	\begin{theoG}\cite{Ban-Kundu_lmj}
		Let	$S_{1}$ be defined as in (\ref{e1.2}), $f$ be a meromorphic function having finitely many poles in $\mathbb{C}$ and $\mathcal{L}$ be a non-constant $L$-function such that $E_{f}(S_{1},s)=E_{\mathcal{L}}(S_{1},s)$. If
		\\(i) $s\geq 2$ and $n>2m+4$, or if
		\\(ii) $s=1$ and $n>2m+5$,
		then $f \equiv \mathcal{L}$.
	\end{theoG}
	\begin{theoH}\cite{Ban-Kundu_lmj}
		Let $S$ be defined as in (\ref{e1.1}), $f$ be a meromorphic function having finitely many poles in $\mathbb{C}$ and let $\mathcal{L}$ be a non-constant
		$L$-function. 
		Suppose $E_{f}(S,s)=E_{\mathcal{L}}(S,s)$ and for some finite $c\not\in S$, $f$, $\mathcal{L}$ share $(c,0)$. Also let $a_{i}\;(i=1,2,\ldots,n-m)$ be the zeros of $nz^{k}+ma$, where $k=n-m\;(\geq 1)$ and denote  $S'=\{a_1,a_2,\ldots,a_{n-m}\}$. \\ {\bf I.} Suppose $c=0$.   
		\\ When {\em(i)} $s\geq 2$, $n>2m+2$ or {\em(ii)} $s=1$, $n>2m+3$ or 
		then  
		we have $f = \mathcal{L}$.\\  {\bf II.} Suppose $c\not =0$.\\	
		{\em(A)} Let $c\in S'$. When $l=n-1$ and {\em(i)} $s\geq 2$, $n> 2k+2$ or {\em(ii)} $s=1$, $n> 2k+3$  \\ then, we have $f \equiv \mathcal{L}$.\\
		{\em(B)} 
		Next let $c\not\in S'$. When {\em(i)} $s\geq 2$, $n>2k+4$ or
		{\em(ii)} $s=1$, $n> 2k+5$ 
		then we have $f \equiv \mathcal{L}$.
	\end{theoH} 
	To proceed further the following observations and relevant discussions are needed.
	\par First from {\it{Lemma 4}} in \cite{Ban-Kundu_lmj}, we know there can be only one multiple zero of $P(w)$ of multiplicity exactly $2$. So as $n-1\leq l\leq n$, when $P(w)$ has multiple zero, then we have $l=n-1$. Let $f$ and a $L$-function $\mathcal{L}$ share the set $S=\{w:w^n+aw^m+b=0\}$ CM.  Throughout the paper let's consider $\alpha_1$ be a multiple zero and $\alpha_2,\ldots,\alpha_{n-1}$ be the simple zeros of $P(w)$. Now as $f$ and $\mathcal{L}$ share the set $(S,\infty)$, then $(f-\alpha_1)(f-\alpha_2)\ldots(f-\alpha_{n-1})$ and $(\mathcal{L}-\alpha_1)(\mathcal{L}-\alpha_2)\ldots(\mathcal{L}-\alpha_{n-1})$ share $(0,\infty)$. But it does not imply $P(f)=(f-\alpha_1)^{2}(f-\alpha_2)\ldots(f-\alpha_{n-1})$ and $P(\mathcal{L})= (\mathcal{L}-\alpha_1)^2(\mathcal{L}-\alpha_2)\ldots(\mathcal{L}-\alpha_{n-1})$ share $(0,\infty)$. Actually, since $f$ and $\mathcal{L}$ share $(S,\infty)$, it follows that if $z_0$ be a zero of $f-\alpha_1$ of multiplicity $p$, then it could be a zero of $\mathcal{L}-\alpha_i$, where $i\in \{1,2,3,\ldots,n-1\}$ of multiplicity $p$. Now if $z_0$ be a zero of $\mathcal{L}-\alpha_i$ ($i\not=1$) of multiplicity $p$, then clearly $P(f)$ has a zero at $z_0$ of multiplicity $2p$ where as the same will be a zero of $P(\mathcal{L})$ of multiplicity $p$. Hence when $P(w)$ has multiple zero, it does not imply that $P(f)$ and $P(\mathcal{L})$ share $(0,\infty)$. It is easy to check whatever may be the sharing condition under which $f$ and $\mathcal{L}$ share the set $S$, $P(f)$ and $P(\mathcal{L})$ always share $(0,0)$.

	\par Assuming that $f$ and $\mathcal{L}$ share  $(S,\infty)$, in the proof of {\it{Theorem 5 }} (\cite{Yan-Li-Yi_Lith}), Yuan-Li-Yi  constructed a function [see (3.10)]  $h=R_1e^{\alpha}=\frac{f^n+af^m+b}{\mathcal{L}^n+a\mathcal{L}^m+b}$, where $R_1$ is a rational function and $\alpha$ is an entire function, i.e., $h$ has finitely many poles and zeros. Again in the proof of {\it{Theorem 6 }}(in \cite{Yan-Li-Yi_Lith}) before the line (3.40), Yuan-Li-Yi used  (3.10) and using this function $h$ and the argument $N(r,h)=N(r,0;h)=O(\log r),\;T(r,h)=O(r)$ [(see (3.20)-(3.23), last line of p. 8, just after (3.43) in p. 11)], Yuan-Li-Yi proved the rest part of their theorems. In the construction of $h$ they actually considered $P(f)$ and $P(\mathcal{L})$ share $(0,\infty)$ but from the previous discussion it is clear that when $P(w)$ has multiple zero, then $f$ and $\mathcal{L}$ share  $(S,\infty)$ does not always imply $P(f)$ and $P(\mathcal{L})$ share $(0,\infty)$. A close inspection will yield, $h=\frac{P(f)}{P(\mathcal{L})}=\frac{(f-\alpha_1)^{2}(f-\alpha_2)\ldots(f-\alpha_{n-1})}{(\mathcal{L}-\alpha_1)^2(\mathcal{L}-\alpha_2)\ldots(\mathcal{L}-\alpha_{n-1})}=\frac{f-\alpha_1}{\mathcal{L}-\alpha_1}.\frac{(f-\alpha_1)(f-\alpha_2)\ldots(f-\alpha_{n-1})}{(\mathcal{L}-\alpha_1)(\mathcal{L}-\alpha_2)\ldots(\mathcal{L}-\alpha_{n-1})}$, and though $(f-\alpha_1)(f-\alpha_2)\ldots(f-\alpha_{n-1})$ and $(\mathcal{L}-\alpha_1)(\mathcal{L}-\alpha_2)\ldots(\mathcal{L}-\alpha_{n-1})$ share $(0,\infty)$, from the factor $\frac{f-\alpha_1}{\mathcal{L}-\alpha_1}$, it follows that $h$ can have infinitely many zeros and poles according as $f-\alpha_1$, $\mathcal{L}-\alpha_1$ have infinitely many zeros. So clearly if $P(w)$ has a multiple zero, then we can not assure the existence of such rational $R_1$ such that $h=R_1e^\alpha$, unless $f$ and $\mathcal{L}$ share the multiple zero of $P(w)$. Hence the results in \cite{Yan-Li-Yi_Lith}, i.e., {\it{ Theorem B}}, {\it{ Theorem C}} are properly correct only when $P(w)$ has simple zero (i.e., $l=n$). 
	
	Unfortunately, at the time of proving the theorems, the present authors \cite{Ban-Kundu_lmj} also carried forwarded the same mistakes  under a different method than that was adopted in \cite{Yan-Li-Yi_Lith}. In the proofs of {\it{Theorems F - H}}, it was assumed that when $f$, $\mathcal{L}$ share $S$ with weight $s\geq 1$, then $F$, $G$ share $(1,s)$. Clearly when $l=n$, i.e., $P(w)\;(P_1(w))$ has only simple zeros, then the argument is absolutely correct. But the problem arises when $P(w)\;(P_1(w))$ contains  multiple zero. In this case, $F$, $G$ need not to share $(1,s)$ as we have already elaborated in the first paragraph of the discussion. So p.7, l.1 in {\it{Lemma 6}}; p.12, l.2 in the proof of {\it{Theorem 1}}; p. 10, l.2, in the proof of {\it{Theorem 2}}  and p.14, l.1 in the proof of {\it{Theorem 3}} in \cite{Ban-Kundu_lmj}  are not true in general.  Therefore the theorems in  \cite{Ban-Kundu_lmj} are partially true for $s\geq 1$. In p.9, l.2 in the proof of {\it{Theorem 1.1 }}of Sahoo-Sarkar's paper \cite{Sahoo-Sarkar_al_i.cuza} the same errors can be found.
	\par  However, for the case of IM sharing of the sets in \cite{Ban-Kundu_lmj}, there is nothing wrong as in this case $F$ and $G$ share $(1,0)$ and $P(f)$ and $P(g)$ share $(0,0)$ and so our results \cite{Ban-Kundu_lmj} are correct irrespective of  the presence of simple or multiple zeros in $P(w)$ or $P_1(w)$. Consequently there were no ambiguity in our results \cite{Ban-Kundu_lmj} for IM sharing and they improve and rectify the results of Sahoo-Halder \cite{{Sahoo-Hal_20_Lith}}.
	
	\par In another previous paper \cite{Kundu-Ban_palermo}, though we mainly focused to solve a question raised by Lin-Lin \cite{Lin-Lin_filomat} but also we re-investigated {\it{ Theorem B}} for $c=0$ as follows: 
	\begin{theoI}\cite{Kundu-Ban_palermo}
		Let $f$ be a meromorphic function in $\mathbb{C}$ with finitely many poles and  $S$ be defined in (\ref{e1.1}). Here $a$, $b$ are two nonzero constants and $n$, $m$  are relatively prime positive integers such that $n>2m$. If $f$ and a non-constant $L$-function $\mathcal{L}$ share $(S,\infty)$ and $(0,0)$, then we will get $\mathcal{L} \equiv f.$
	\end{theoI}
	\par As {\it{ Theorem I}}, was based on the further investigations on {\it{ Theorem B}}, the same mistake was done in {\it{ Theorem I}} and so the same holds only when $P(w)$ has simple zeros. Not only that in the proof of {\it{ Theorem I}}, at the time of estimating the zeros of $\frac{\hat{H}}{Q}-1$, mistakenly we skipped the part when $\frac{\hat{H}}{Q}=1$, [( see p.13, equation (3.3) \cite{Kundu-Ban_palermo})]. The detail explanation when $\frac{\hat{H}}{Q}=1$ is given in the proof of {\it{ Theorem \ref{t1.3}}} stated later on. 
	
	\par Clearly the above series of results fail to hold for $l=n-1$, as in this case all the authors considered the same mistakes. So at this present juncture, it is high time to rectify all of them under insertion of minimal set of conditions. To this end, we are now going to introduce the definition of  restricted weighted sharing of sets which can be obtained from the traditional definition of weighted sharing of sets with some subtle modifications. With the help of this new definition, we will be able to overcome the mistakes in the above theorems. That is to say, we would like to re-investigate the nature of sharing when the generating polynomial of the set, which is not a PURSM, contains a multiple zero. 
	
	\begin{defi}
		Let us consider a set $\hat{S}=\{a_1,a_2,\ldots,a_p\}\subset \mathbb{C}$ and $f$, $g$ be two non-constant meromorphic functions in $\mathbb{C}$. Let $\hat{S}=S_1\cup S_2\ldots\cup S_r$, where $r$ is a positive integer and $S_i\cap S_j=\phi\;(1\leq i, j\leq r;i\not=j)$. We say $f$ and $g$ share the set $\hat{S}$ with weight $k$ in a restricted manner if $f$, $g$ share	$(S_1,k_1), (S_2,k_2),\ldots,(S_r,k_r)$ in usual way as defined in \cite{Lahiri_Nagoya(01)}, where $k=\min\{k_1,k_2,\ldots,k_r\}$. For the sake of convenience we will denote the restricted sharing of $\hat{S}$ with weight $k$ by  $(\hat{S},k)^{*}=\{(S_1,k_1), (S_2,k_2),\ldots,(S_r,k_r)\}$. Clearly $(S,k)^*$ always implies $(S,p)^*$ for all $p\leq k$. When $k=\infty$ or $0$, we say $f$, $g$ share the set restricted CM or IM. Naturally $r=1$ the definition coincides with that of the traditional one.  
	\end{defi}
	
	\par With the help of the definition of restricted sharing we can obtain the following set of results easily, which rectifies a series of erroneous results. Since the results in \cite{Ban-Kundu_lmj}, \cite{Sahoo-Sarkar_al_i.cuza}, \cite{Yan-Li-Yi_Lith} are all correct when $l=n$ i.e., $P(w)$ ($P_1(w)$) has all simple zeros, in the following theorems and corollaries we only consider $P(w)$ has multiple zero and $S$ has $l=n-1$ elements.\par As {\it{Theorem F}} is the generalization of {\it{Theorem E}} and improvement of {\it{Theorem C}}, using restricted set sharing
	we will rectify only {\it{Theorem F}} which in turn overcome the erroneous issues in  {\it{Theorem C}} and {\it{Theorem E}}.
	
	\begin{theo}\label{t1.1}
		Let $S$ be defined as in (\ref{e1.1}). Also let $f$ be a meromorphic function having finitely many poles in $\mathbb{C}$ and $\mathcal{L}$ be a non-constant $L$-function and $l=n-1$. If  \\(i) $f$ and $\mathcal{L}$ share $(S,1)^*=\{(\alpha_{1},1),(\{\alpha_2,\ldots,\alpha_{n-1}\},2)\}$ and $n>2m+4$; or if \\(ii) $f$ and $\mathcal{L}$ share $(S,0)^*=\{(\alpha_{1},0),(\{\alpha_2,\ldots,\alpha_{n-1}\},1)\}$ and $n>2m+5$, \\then we will get $f\equiv \mathcal{L}$.
	\end{theo}
	\par Analogous result corresponding to {\it{Theorem \ref{t1.1}}}, can be obtained for the set $S_1$.
	\par Next we are going to present the corrected form of {\it{Theorem H}},  which will immediately rectify {\it{Theorem B}} and {\it{Theorem D}}.
	\begin{theo}\label{t1.2}Let $S$ be defined as in (\ref{e1.1}), $l=n-1$ and $S'$ be defined as in {\it{Theorem H}}. Also let $f$ be a meromorphic function having finitely many poles in $\mathbb{C}$, $\mathcal{L}$ be a non-constant
		$L$-function. 
		For two positive integers $k_1$, $k_2$ let $f$, $\mathcal{L}$ share
		(i) $(S,1)^*=\{(\alpha_{1},1),(\{\alpha_2,\ldots,\alpha_{n-1}\},k_1)\}$, (ii)$(S,0)^*=\{(\alpha_{1},0),(\{\alpha_2,\ldots,\alpha_{n-1}\},k_2)\}$  and for some finite $c\not\in S$, $f$, $\mathcal{L}$ share $(c,0)$.  
		\\ {\bf I.} Suppose $c=0$.   
		\\ When {\em(i)} $k_1\geq 2$, $n>2m+2$ or {\em(ii)} $k_2\geq 1$, $n>2m+3$; then  
		we have $f\equiv\mathcal{L}$.\\  {\bf II.} Suppose $c\not =0$.\\	
		{\em(A)} Let $c\in S'$. If {\em(i)} $k_1\geq 2$, $n> 2k+2$ or {\em(ii)} $k_2\geq 1$, $n> 2k+3$;  then we have $f\equiv \mathcal{L}$.\\
		{\em(B)} 
		Next let $c\not\in S'$. If {\em(i)} $k_1\geq 2$, $n>2k+4$ or
		{\em(ii)} $k_2\geq 1$, $n> 2k+5$, 
		then we have $f \equiv \mathcal{L}$
		
	\end{theo}	
	For $l=n-1$, the corrected form of {\it{Theorem I}} can be restated as follows. 
	\begin{theo}\label{t1.3}
		Under the same situation as in {\it{Theorem I}}, if $l=n-1$ and $f$, $\mathcal{L}$ share $(S,\infty)^*=\{(\alpha_{1},\infty),(\{\alpha_2,\ldots,\alpha_{n-1}\},\infty)\}$, 
		then we will have $f\equiv \mathcal{L}$.
	\end{theo}
	Next we turn our attention on some useful application of restricted sharing of sets.
	\par Let us consider the Frank-Reinders polynomial  \cite{Frank-Reinders_comp.var}:
	
	\bea\label{e1.3} P_{FR}(z)=\frac{(n-1)(n-2)}{2}z^{n}-n(n-2)z^{n-1}+\frac{n(n-1)}{2}z^{n-2}-c .\eea When $c\not=0,1$, then for $n\geq11$, it is well known from \cite{Frank-Reinders_comp.var} that the set
	$S=\{z:P_{FR}(z)=0\}$ is a URSM. Here, if we take $c=1$ then $z-1$ becomes a multiple factor of  $P_{FR}(z)$. Also we know from \cite{P.Li-C.C.Yang_kodai(95)} that the polynomial generating URSM must contain distinct zeros. In fact, till date to construct the traditional  unique range sets, when the representations of the same was based on the zeros of the generating polynomial, then it was always assumed that the generating polynomial contains simple zeros. So it will be interesting to investigate  for $c=1$ whether the set $S$ becomes a URSM or not. In this respect, next we will show that restricted sharing of sets become an important tool to tackle the situation. \par For the sake of convenience, let
	\bea\label{e1.4}\;\;\;\;\;\;\;\;\;\; P^1_{FR}(z)=\frac{(n-1)(n-2)}{2}z^{n}-n(n-2)z^{n-1}+\frac{n(n-1)}{2}z^{n-2}-1=(z-1)^3Q_{n-3}(z),\eea and $S_Q=\{z:Q_{n-3}(z)=0\}$, where $Q_{n-3}(z)$ has $n-3$ distinct zeros. When $P^{1}_{FR}(z)$ has a multiple zero then with respect to the traditional definition of set sharing of $S(=\{z:P^1_{FR}(z)=0\})$, no conclusion about the uniqueness of $f$ and $\mathcal{L}$ can be derived. But with the aid of the newly introduced definition of restricted set sharing on $S$ the uniqueness relation between $f$ and $\mathcal{L}$ can be obtained for the case $c=1$.
	\par So we have the following result.
	\begin{theo}\label{t1.4}
		Let $S=\{z:P^1_{FR}(z)=0\}$, $f$  be a meromorphic function having finitely many poles in $\mathbb{C}$  and let $\mathcal{L}$ be a non-constant $L$-function. If \\(i)  $f$ and $\mathcal{L}$ share $(S,1)^*=\{(1,1),(S_Q,3)\}$ and $n>7$; or if \\(ii)  $f$ and $\mathcal{L}$ share $(S,0)^*=\{(1,0),(S_Q,2)\}$ and $n>8,$ then we will get $f\equiv \mathcal{L}$.
	\end{theo}
	From equation (\ref{e1.4}) and {\it Theorem \ref{t1.4}} we can deduct the following corollaries:
	\begin{cor}
		Let $S=\{z:P^1_{FR}(z)=0\}$. If $f$ and $\mathcal{L}$ share $(S,1)^*=\{(1,1),(S_Q,s)\}$ where $s\geq 3$  and it contains $6$ elements, then we get $f\equiv \mathcal{L}$.
	\end{cor}
	\begin{cor}
		If $f$ and $\mathcal{L}$ share $S=\{z:P^1_{FR}(z)=0\}$ restricted IM such that $(S,0)^*=\{(1,0),(S_Q,s)\}$ where $s\geq 2$, and it contains $7$ elements, then we get $f\equiv \mathcal{L}$.
	\end{cor}
	In the next section, for some meromorphic functions $f$ and $g$, we are going to use some known counting functions like $N_{E}^{1)}(r,1;f\vline=1)$, $\ol N(r,1;f\mid=1)$, $N_{E}^{1)}(r,1;f\vline\geq s)\;(s\geq 1)$, $\ol N_*(r,1;f,g)$, $\ol N(r,a;f\mid\geq s;g\not=b)$ and $\ol N(r,0;f'\mid\geq s; f=b)$ without explanations as those are very common (see \cite{Ban-Kundu_lmj}).
	\section{lemmas}
	Now we need the following lemmas to proceed further.
	Let $F$ and $G$ be two non-constant meromorphic functions defined in $\mathbb{C}$ as follows
	\bea\label{e2.1}\;\;\;\;\;\;\;F=\frac{(n-1)(n-2)f^{n-2}(f^2-\frac{2n}{n-1}f+\frac{n}{n-2})}{2c}\;,\;\;G=\frac{(n-1)(n-2)g^{n-2}(g^2-\frac{2n}{n-1}g+\frac{n}{n-2})}{2c},\eea and in ({\ref{e2.1}}), when $c=1$, we will take $F=F_1$, $G=G_1$.\par
	Henceforth we shall denote by $H$ the following function. \be{\label{e2.2}}H=\left(\frac{F^{''}}{F^{'}}-\frac{2F^{'}}{F-1}\right)-\left(\frac{G^{''}}{G^{'}}-\frac{2G^{'}}{G-1}\right).\ee  
	
	\begin{lem}\label{l2.1}\cite{Yi-Lu} Let $F$, $G$ be two non-constant meromorphic functions sharing $(1,0)$ and $H\not\equiv 0$. Then $$N_{E}^{1)}(r,1;F\vline=1)=N_{E}^{1)}(r,1;G\vline=1)\leq N(r,\infty;H)+S(r,F)+S(r,G).$$\end{lem}
	\begin{lem}\label{l2.2}
		Let $f$, $g$ share $S=\{z:P_{FR}(z)=0\}$ IM, i.e.,  $F$, $G$ share $(1,0)$ with $H\not\equiv0$, then 
		\beas (i)\; N(r,\infty;H)&\leq&\ol N(r,0;f)+\ol N(r,\infty;f)+\ol N(r,1;f)\nonumber+\ol N(r,0;g)+\ol N(r,\infty;g)\\ &&+\ol N(r,1;g)+\ol N_{*}(r,1;F,G)+\ol N_{0}(r,0;f^{'})+\ol N_{0}(r,0;g^{'})\eeas and if $f$, $g$ share $S=\{z:P^1_{FR}(z)=0\}$ restricted IM such that $(S,0)^*=\{(1,0), (S_Q,0)\}$, i.e., $F_1$, $G_1$ share $(1,0)$ and $H\not\equiv 0$ then	\beas  (ii)\;N(r,\infty;H)&\leq&\ol N(r,0;f)+\ol N(r,\infty;f)+\ol N(r,0;g)+\ol N(r,\infty;g)+\ol N_{*}(r,1;F_1,G_1)\\ &&+\ol N_{1}(r,0;f^{'})+\ol N_{1}(r,0;g^{'}),\eeas 
		where $\ol N_{0}(r,0;f^{'})\;(\ol N_{1}(r,0;f^{'}))$ is the reduced counting function of those zeros of $f^{'}$ which are not the zeros of $f(f-1)(F-1)\;(f(F_1-1))$ and $\ol N_{0}(r,0;g^{'})\;(\ol N_{1}(r,0;g^{'}))$ is similarly defined.\end{lem}
	\begin{proof}
		Clearly from the given condition, \beas F'=\frac{n(n-1)(n-2)f^{n-3}(f-1)^2f^{'}}{2c}\;\;\text{ and}\;\;G'=\frac{n(n-1)(n-2)g^{n-3}(g-1)^2g^{'}}{2c}.\eeas By simple calculations proof of (i) can be carried out easily.\par   From (\ref{e1.4}) we know $F_1-1=(f-1)^3Q_{n-3}(f)$ and $G_1-1=(g-1)^3Q_{n-3}(g)$. Since here $f$ and $g$ share $(1,0)$ the $1$ points of $f$ and $g$ with different multiplicities contribute towards the poles of $H$ and will be accommodated within $\ol N_*(r,1;F_1,G_1)$.
		Hence we can get the proof of (ii) easily.
	\end{proof}
	\begin{lem}\label{l2.3}\cite{Mokhon'ko-lem} Let $f$ be a non-constant meromorphic function and let \[R(f)=\frac{\sum\limits _{k=0}^{n} a_{k}f^{k}}{\sum \limits_{j=0}^{m} b_{j}f^{j}}\] be an irreducible rational function in $f$ with constant coefficients $\{a_{k}\}$ and $\{b_{j}\}$, where $a_{n}\not=0$ and $b_{m}\not=0$. Then $$T(r,R(f))=dT(r,f)+S(r,f),$$ where $d=\max\{n,m\}$.\end{lem}
	\begin{lem}\label{l2.4}
		Let $f$ be a meromorphic function having finitely many poles in $\mathbb{C}$ and $S=\{z:P_{FR}(z)=0\}\;(or\;S=\{z:P^1_{FR}(z)=0\})$. If $f$ and a non constant $L$-function $\mathcal{L}$ share the set $S$ IM (or restricted IM), then $\rho(f) =\rho(\mathcal{L}) = 1.$
	\end{lem}
	
	\begin{proof}
		We omit the proof as the same can be found out in the proof of Theorem 5, \{p. 6, \cite{Yan-Li-Yi_Lith}\}.
	\end{proof}
	Next we omit the proofs of {\it{Lemma \ref{l2.5}, \ref{l2.6}}}, stated below, as the same can be obtained in the line of proof of {\it{Lemma 6, 7}} in \cite{Ban-Kundu_lmj}.
	\begin{lem}\label{l2.5}
		Let $F=-\frac{f^{n}}{af^{m}+b}$ and $G=-\frac{g^{n}}{ag^{m}+b}$, where $f$ and $g$ be any two non-constant meromorphic functions, 
		$n$, $m$ be relatively prime positive integers, such that $n > m \geq 1$ and $a$, $b$ be non-zero finite constants. Let $\gamma_{i}$, $i=1,2,\ldots,m$ be the roots of $aw^{m}+b=0$. If $H\not\equiv 0$ and $F$, $G$ share $(1,s)$ then, 
		\beas\frac{n}{2}\big(T(r,f)+T(r,g)\big)&\leq& 2\big(\ol N(r,0;f)+\ol N(r,0;g)+\ol N(r,\infty;f)+\ol N(r,\infty;g)\big)+\sum_{i=1}^{m}\big(N_{2}(r,\gamma_{i};f)\\&&+N_{2}(r,\gamma_{i};g)\big)+\left(\frac{3}{2}-s\right)\ol N_{*}(r,1;F,G)+S(r,f)+S(r,g).\eeas
	\end{lem}
	\begin{lem}\label{l2.6}
		Let $F = -(f^{n} + af^{n-m})/b$ and $G = -(g^{n} + ag^{n-m})/b$, where $f$ and $g$ be any two non-constant meromorphic functions 
		and	$n$, $m$ be relatively prime positive integers such that $n > m \geq 1$, and $a$, $b$ be non-zero finite constants. Let $\delta_{i},\;i=1,2,\ldots,m$ be the distinct roots of the equation $w^{m}+a = 0$. If $H \not\equiv 0$ and $F,G$ share $(1,s)$, then,  \beas\frac{n}{2}\big(T(r,f)+T(r,g)\big)&\leq& 2\big(\ol N(r,0;f)+\ol N(r,0;g)+\ol N(r,\infty;f)+\ol N(r,\infty;g)\big)+\sum_{i=1}^{m}\big(N_{2}(r,\delta_{i};f)\\&&+N_{2}(r,\delta_{i};g\big)+\left(\frac{3}{2}-s\right)\ol N_{*}(r,1;F,G)+S(r,f)+S(r,g).\eeas\end{lem}

	\section{proofs of the theorems}
	\begin{proof}[Proof of Theorem \ref{t1.1} and Theorem \ref{t1.2}]
		It is given that when $l=n-1$ then $f$, $\mathcal{L}$ share $(S,1)^*=\{(\alpha_{1},1),(\{\alpha_2,\ldots,\alpha_{n-1}\},2)\}$. Then clearly $P(f)=(f-\alpha_{1})^2(f-\alpha_{2})\ldots(f-\alpha_{n-1})$ and $P(\mathcal{L})=(\mathcal{L}-\alpha_1)^2(\mathcal{L}-\alpha_{2})\ldots(\mathcal{L}-\alpha_{n-1})$ share $(0,2)$ CM and similarly $F=-\frac{f^n}{af^m+b}$ and $G=-\frac{\mathcal{L}^n}{a\mathcal{L}^m+b}$ share $(1,2)$. Similarly when $f$, $\mathcal{L}$ share  $(S,0)^*=\{(\alpha_{1},0),(\{\alpha_2,\ldots,\alpha_{n-1}\},1)\}$. Then clearly $P(f)$, $P(\mathcal{L})$ share $(0,1)$, i.e., $F$, $G$ share $(1,1)$. \par  Now instead of {\it{Lemma 6}} in \cite{Ban-Kundu_lmj} using {\it{Lemma \ref{l2.6}}}, for $s=1$, $2$  and proceeding similarly as done in the proof of {\it{Theorem 1}}, {\it{Theorem 3}} in \cite{Ban-Kundu_lmj} for $l=n-1$, we can get the result.	
	\end{proof}
	\begin{proof}[Proof of Theorem \ref{t1.3}]  Let $l=n-1$, $f$ and $\mathcal{L}$ share $(S,\infty)^*=\{(\alpha_1,\infty),(\{\alpha_2,\ldots,\alpha_{n-1}\},\infty)\}$. It follows that $P(f)$ and $P(\mathcal{L})$ share $(0,\infty)$ and hence $F=-\frac{f^{n}}{af^{m}+b}$ and $G=-\frac{\mathcal{L}^{n}}{a\mathcal{L}^{m}+b}$ share $(1,\infty)$. 
		\par First we introduced the auxiliary function $$\frac{P(f)}{P(\mathcal{L})}=H_o.$$ Since $P(f)$ and $P(\mathcal{L})$ share $(0,\infty)$ and $f,\mathcal{L}$ have finitely many poles, we can find some rational $Q_o$ such that $$\frac{P(f)}{P(\mathcal{L})}=Q_o.e^{q},$$  where $q$ is a entire function with $deg(q)\leq 1$.\\ Again if, \beas P(f)&=&P(\mathcal{L})\\f^n-\mathcal{L}^n&=&-a(f^m-\mathcal{L}^m)\\\mathcal{L}^{n-m}&=&-\frac{a(h^m-1)}{h^n-1},\eeas where $h=\frac{f}{\mathcal{L}}$. Since $f$ and ${\mathcal{L}}$ share $(0,0)$, then from $P(f)=P(\mathcal{L})$ we get $f$, ${\mathcal{L}}$ share $(0,\infty)$ and ($\infty$,$\infty$). Clearly then $h$ has no zero no poles and since ${\mathcal{L}}$ can have at most one pole then $h$ must have $n-2\;(n\geq 3)$ exceptional value, a contradiction. Hence one must have $h\equiv 1\implies f\equiv \mathcal{L}$. \par If $P(f)\not=P(\mathcal{L})$ then we have $\ol N(r,0;f)=\ol N(r,0;\mathcal{L})\leq \ol N(r,1;Q_o.e^q)\leq O(r)$. Now proceeding same as in l.8, p.13 in \cite{Kundu-Ban_palermo} we will get a contradiction and the remaining part of the proof is same as {\it{Theorem 1.1}} in \cite{Kundu-Ban_palermo}. So we omit the detail.
	\end{proof}
	\begin{note}
		When $l=n$, in {\it{Theorem 1.1}} \cite{Kundu-Ban_palermo},  $\frac{\hat{H}}{Q}=1\implies P(f)=P(\mathcal{L})$, then proceeding similarly as in above we will get $f\equiv \mathcal{L}$.
	\end{note}
	\begin{proof}[Proof of Theorem \ref{t1.4}]
		First let us consider $$\hat{F}=\frac{(n-1)(n-2)f^{n-2}(f^2-\frac{2n}{n-1}f+\frac{n}{n-2})}{2}\;,\;\;\hat{G}=\frac{(n-1)(n-2)\mathcal{L}^{n-2}(\mathcal{L}^2-\frac{2n}{n-1}\mathcal{L}+\frac{n}{n-2})}{2},$$ and consider the same function $H$ for $F=\hat{F}$, $G=\hat{G}$. 
		\\{\bf\underline{Case I.}} It is given that $f$ and $\mathcal{L}$ share $S=\{z:P_{FR}^1(z)=0\}$ with weight $1$ in a restricted manner such that $(S,1)^*=\{(1,1),(S_Q,3)\}$. Here we have $\hat{F}-1=(f-1)^3Q_{n-3}(f)$, $\hat{G}-1=(\mathcal{L}-1)^3Q_{n-3}(\mathcal{L})$. As from the definition we have $f$ and $\mathcal{L}$ share $(1,1)$, if $z_0$ is a common zero of $f-1$ and $\mathcal{L}-1$ of multiplicity $t(\geq 1)$, then the contribution of those $1$ points in $\hat{F}-1$ and $\hat{G}-1$ is of multiplicity $\geq 3t$ and so the difference in multiplicities can occur only when $t\geq 2$.  Clearly $\hat{F}$, $\hat{G}$ share $(1,3)$. And here $\ol N(r,\infty;f)=\ol N(r,\infty;\mathcal{L})=O(\log r)$.
		\par Again, as $\hat{F}$ and $\hat{G}$ share $(1,3)$ then $ N_{E}^{1)}(r,1;\hat{F}\mid=1)=\ol N(r,1;\hat{F}\mid=1)=\ol N(r,1;\hat{G}\mid=1)$. Also from {\it{Lemma \ref{l2.4}}}, we know that $f$, $\mathcal{L}$ are of finite order and so $S(r,f)=S(r,\mathcal{L})=O(\log r)$. \par Now note that, 
		\beas&& \ol N(r,0;\hat{G}'\mid \hat{G}=1)+\ol N_1(r,0;\mathcal{L}')\leq \ol N(r,1;\mathcal{L})+\ol N(r,0;\mathcal{L}'\mid \mathcal{L}\not=0)\\&\leq& \ol N(r,1;\mathcal{L})+\ol N(r,0;\mathcal{L})+\ol N(r,\infty;\mathcal{L})\leq \ol N(r,1;\mathcal{L})+\ol N(r,0;\mathcal{L})+O(\log r).\eeas Similarly we can get, \beas N(r,0;\hat{F}'\mid \geq 3;\hat{F}=1) \leq \frac{1}{3}\left(N(r,1;f)+\ol N(r,0;f)\right)+O(\log r).\eeas
		
		Using the above relations and {\it{Lemma \ref{l2.1}}}, the second part of {\it{Lemma \ref{l2.2}}}
		and the Second Fundamental Theorem we have \bea\label{e3.1}&& (n-2)T(r,f)\\&\leq& \ol N(r,0;f)+\ol N(r,\infty;f)+\ol N(r,1;\hat{F})- N_1(r,0;f')\nonumber\\\nonumber&\leq&	\ol N(r,0;f)+\ol N(r,\infty;f)+\ol N(r,1;\hat{F}\mid=1)+\ol N(r,1;\hat{F}\mid\geq 2)-N_1(r,0;f')\\\nonumber&\leq& 2\ol N(r,0;f)+\ol N(r,0;\mathcal{L})+\ol N_*(r,1;\hat{F},\hat{G})+\ol N(r,1;\hat{G}\mid\geq 2)+\ol N_1(r,0;\mathcal{L}')\\\nonumber&&+O(\log r)\\\nonumber&\leq&2\ol N(r,0;f)+\ol N(r,0;\mathcal{L})+\ol N(r,1;\hat{F}\mid\geq 4)+\ol N(r,1;\hat{G}\mid\geq 2)+\ol N_1(r,0;\mathcal{L}')+O(\log r)\\\nonumber&\leq& 2\ol N(r,0;f)+\ol N(r,0;\mathcal{L})+\ol N(r,0;\hat{F}'\mid\geq 3;\hat{F}=1)+\ol N(r,0;\hat{G}'\mid \hat{G}=1) +\ol N_1(r,0;\mathcal{L}')\\\nonumber&&+O(\log r)
		\\\nonumber&\leq&2\ol N(r,0;f)+2\ol N(r,0;\mathcal{L})+\ol N(r,1;\mathcal{L})+\frac{1}{3}(\ol N(r,0;f)+N(r,1;f))+O(\log r).\eea Proceeding similarly we can get \bea\label{e3.2}\;\;\;(n-2)T(r,\mathcal{L})&\leq& 2\ol N(r,0;f)+2\ol N(r,0;\mathcal{L})+\ol N(r,1;f)+\frac{1}{3}(\ol N(r,0;\mathcal{L})+N(r,1;\mathcal{L}))\\\nonumber&&+O(\log r).\eea Hence adding (\ref{e3.1}) and (\ref{e3.2})  we get \beas (n-2)\{T(r,f)+T(r,\mathcal{L})\}\leq \left(5+\frac{2}{3}\right)\{T(r,f)+T(r,\mathcal{L})\}+O(\log r).\eeas Clearly for $n\geq 8$ we get a contradiction and hence $H\equiv 0$. \par Then from (\ref{e2.2}) we get \bea \label{e3.3} \frac{1}{\hat{F}-1}\equiv \frac{A}{\hat{G}-1}+B,\eea  where $A(\not=0)$ and $B$ are two constants. So in view of {\it{Lemma \ref{l2.3}}}, from (\ref{e3.3}) we get
		\be\label{e3.4} T(r,f)=T(r,\mathcal {L})+S(r,f).\ee
		\par Suppose $B\not=0$. Then from (\ref{e3.3}) we get \be\label{e3.5} \hat{F}-1\equiv \frac{\hat{G}-1}{B\hat{G}+A-B}.\ee 
		\par
		If $A-B\not =0$, then noting that $\frac{B-A}{B}\not =1$, from (\ref{e3.1}) we get 
		$$\ol N\left (r,\frac{B-A}{B};\hat{G}\right)=\ol N(r,\infty;\hat{F}).$$
		\par Therefore in view of {\it{Lemma \ref{l2.3}}} and equation (\ref{e3.3}) using the Second Fundamental Theorem  we have 
		\beas T(r,\hat{G})=nT(r,\mathcal{L})&\leq&\ol N(r,0;\hat{G})+\ol N(r,\infty;\hat{G})+\ol N\left (r,\frac{B-A}{B};\hat{G}\right)+S(r,\hat{G})\\&\leq&3T(r,\mathcal{L})+\ol N(r,\infty;\mathcal{L})+\ol N(r,\infty;f)+O(\log r),\eeas
		which is a contradiction for $n\geq 8$.
		\par If $A-B=0$, then from (\ref{e3.3}) we have \be\label{e3.6}\frac{\hat{G}-1}{\hat{F}-1}  \equiv B\hat{G}=B\frac{(n-1)(n-2)\mathcal{L}^{n-2}(\mathcal{L}^2-\frac{2n}{n-1}\mathcal{L}+\frac{n}{n-2})}{2}.\ee  Clearly (\ref{e3.6}) implies that zeros of $\mathcal{L}$ and $(\mathcal{L}^2-\frac{2n}{n-1}\mathcal{L}+\frac{n}{n-2})$ are poles of $\hat{F}$. Now clearly the zeros of $(z^2-\frac{2n}{n-1}z+\frac{n}{n-2})$ are simple. Let $\xi_i\;(i=1,2)$ be a zero of $(z^2-\frac{2n}{n-1}z+\frac{n}{n-2})$. Suppose $z_0$ be a zero of  $\mathcal{L}-\xi_i$ of multiplicity $p$ then from (\ref{e3.6}) we know that it is a pole of $f$ of multiplicity $q$ $\geq1$ such that $p=nq$. i.e., $p\geq n$. Similarly any zero of $\mathcal{L}$ of multiplicity $r$ is a pole of $f$ of multiplicity $s(\geq 1)$. i.e., $r(n-2)=sn$. It follows that, $r=\frac{sn}{n-2}> 1$. Now using the Second Fundamental Theorem we get 
		\beas 2T(r,\mathcal{L})&\leq& \ol N(r,0;\mathcal{L})+\ol N(r,\infty;\mathcal{L})+\sum_{i=1}^{2}\ol N(r,\xi_i;\mathcal{L})\\&\leq&\left (\frac{1}{2}+\frac{2}{n}\right)T(r,\mathcal{L})+O(\log r),\eeas a contradiction for $n>7$. \par Hence $B=0$. Then from (\ref{e3.3}) we get that $$\hat{G}-1=A(\hat{F}-1).$$ i.e., \be\label{e3.7}\hat{G}=A\left (\hat{F}-1+\frac{1}{A}\right).\ee \par If $A\not=1$ then $(\hat{F}-1+\frac{1}{A})$ can be written as $\left(\hat{F}-1+\frac{1}{A}\right)=\prod_{i=1}^{n}(f-\mu_i)$, then using the Second Fundamental Theorem we have \beas (n-1)T(r,f)&\leq& \sum_{i=1}^{n}\ol N(r,\mu_i;f)+\ol N(r,\infty;f)\leq 3T(r,\mathcal{L})+O(\log r),\eeas which in view of $n>7$ and (\ref{e3.4}) we get a contradiction again. \par Therefore $A=1$, and hence $$\hat{F}=\hat{G}.$$ i.e.,  \bea \label{e3.8}  &&\frac{(n-1)(n-2)}{2}(f^n-\mathcal{L}^n) -n(n-2)(f^{n-1}-\mathcal{L}^{n-1})+\frac{(n-1)n}{2}(f^{n-2}-\mathcal{L}^{n-2})=0.\eea Now proceeding in the same way as done in the last part of main {\it{Theorem }} in \cite{Frank-Reinders_comp.var} we can get $f\equiv \mathcal{L}$.
		\\{\bf\underline{Case II.}} Suppose that $f$ and $\mathcal{L}$ share $S=\{z:P_{FR}^1(z)=0\}$ with weight $0$ in a restricted manner such that $(S,0)^*=\{(1,0),(S_Q,2)\}$. Then as explained in Case I, we have $\hat{F}$, $\hat{G}$ share $(1,2)$.
		Now using {\it{Lemma \ref{l2.1}, \ref{l2.2}}}	and the Second Fundamental Theorem we have \bea\label{e3.9}&& (n-2)T(r,f)\\&\leq& \ol N(r,0;f)+\ol N(r,\infty;f)+\ol N(r,1;\hat{F})- N_1(r,0;f')\nonumber\\\nonumber&\leq&	\ol N(r,0;f)+\ol N(r,\infty;f)+\ol N(r,1;\hat{F}\mid=1)+\ol N(r,1;\hat{F}\mid\geq 2)- N_1(r,0;f')\\\nonumber&\leq& 2\ol N(r,0;f)+\ol N(r,0;\mathcal{L})+\ol N_*(r,1;\hat{F},\hat{G})+\ol N(r,1;\hat{G}\mid\geq 2)+\ol N_1(r,0;\mathcal{L}')\\\nonumber&&+O(\log r)\\\nonumber&\leq&2\ol N(r,0;f)+\ol N(r,0;\mathcal{L})+\ol N(r,1;\hat{F}\mid\geq 3)+\ol N(r,1;\hat{G}\mid\geq 2)+\ol N_1(r,0;\mathcal{L}')+O(\log r)\\\nonumber&\leq&2\ol N(r,0;f)+2\ol N(r,0;\mathcal{L})+\ol N(r,1;\mathcal{L})+\frac{1}{2}(\ol N(r,0;f)+N(r,1;f))+O(\log r).\eea Proceeding similarly we can get, \bea\label{e3.10}\;\;\;\;\;(n-2)T(r,\mathcal{L})&\leq& 2\ol N(r,0;f)+2\ol N(r,0;\mathcal{L})+\ol N(r,1;f)+\frac{1}{2}(\ol N(r,0;\mathcal{L})+N(r,1;\mathcal{L}))\\\nonumber&&+O(\log r).\eea Hence from (\ref{e3.9}) and (\ref{e3.10}) we get \beas (n-2)T(r)\leq 6T(r),\eeas where $T(r)=T(r,f)+ T(r,\mathcal{L})$. Clearly for $n\geq 9$ we get a contradiction.
		\par Now adopting the same procedure as used in (\ref{e3.3})-(\ref{e3.8}), and thereafter in Case I. we can get $f\equiv \mathcal{L}$.	
	\end{proof}
	
	\begin{note}
		Consider the set $S=\{z:P_{FR}(z)=0\}$ and assume $f$, $\mathcal{L}$ share $S$ CM. Now proceeding in the same way as done in (\ref{e3.1})-(\ref{e3.7}) and using first part of  {\it{Lemma \ref{l2.2}}},  for $n\geq 7$, one can get $f\equiv \mathcal{L}$. Here we note that the set contains at least $7$ elements.
	\end{note}

\end{document}